\documentclass[11pt,a4paper,%draft
]{amsart}
\usepackage[utf8]{inputenc}
\usepackage[T1]{fontenc} %font encoding in the output, especially for ä,ö,ü,|,>,<
\usepackage[english]{babel}
\usepackage{amsmath}
\usepackage{amsfonts}
\usepackage{amssymb}
\usepackage{mathtools}
\usepackage{amsthm}
\usepackage{tikz-cd}
\usepackage{float}
\usepackage{verbatim}

\usepackage{todonotes}

\usepackage[colorlinks=true]{hyperref}
\usepackage{latexsym}    % use all LaTeX fonts
\usepackage{amssymb}   % use all AMS fonts
\usepackage{amsmath}    % include some AMS-LaTeX functionality
\usepackage{amsbsy}
\usepackage{amsthm}
\usepackage{amsgen}
\usepackage{amsfonts}
\usepackage{array}
\usepackage{lmodern}
\usepackage[all]{xy}    % Xy-Pic for diagrams  (disabled)
\usepackage{csquotes}

\usepackage{caption}
\usepackage[labelformat=simple]{subcaption}
\newsavebox{\largestimage}

\usepackage{tikz}
\usetikzlibrary{knots}
\usetikzlibrary{decorations.pathmorphing}
\usetikzlibrary{calc, decorations.markings}

\usepackage{color}
\usepackage{verbatim}
\usepackage{url}
\usepackage{enumerate}
\numberwithin{figure}{section}

\tikzset{snake it/.style={decorate, decoration=snake}}

\usepackage{lineno}
\DeclareMathOperator{\Diff}{Diff}

\DeclareMathOperator{\Iso}{Isom}

\DeclareMathOperator{\Ric}{\mathrm{Ric}}

\newtheorem{theorem}{Theorem}
\newtheorem{corollary}[theorem]{Corollary}  % Corollary of Principal Results

		\newtheorem{thm}{Theorem}[section]
		
		\newtheorem{lem}[thm]{Lemma}

		\newtheorem{prop}[thm]{Proposition}
	\theoremstyle{definition}
		\newtheorem{remark}[thm]{Remark}

 \newtheoremstyle{TheoremNum}
        {\topsep}{\topsep}              %%% space between body and thm
        {\itshape}                      %%% Thm body font
        {}                              %%% Indent amount (empty = no indent)
        {\bfseries}                     %%% Thm head font
        {.}                             %%% Punctuation after thm head
        { }                             %%% Space after thm head
        {\thmname{#1}\thmnote{ \bfseries #3}}%%% Thm head spec
    \theoremstyle{TheoremNum}

\numberwithin{equation}{section}

\theoremstyle{definition}

\newtheorem*{ack}{Acknowledgements}
\newtheorem{ex}[thm]{Example}

\usepackage{tikz}
\usepackage{hyperref}
\hypersetup{
    colorlinks=true,
    linkcolor=blue,
    citecolor=blue,
    urlcolor=blue,
    pdfauthor={Diego Corro and  Karla Garcia and Martin Guenter and Jan-Bernhard Kordass},
    pdftitle={Quarter pinched Spherical space forms bundles and Ricci flow}
}

\newcommand{\R}{\mathbb{R}} %Reals
\newcommand{\Z}{\mathbb{Z}} %Integers
\DeclareMathOperator{\Met}{Met}

\DeclareMathOperator{\SO}{SO}

\begin{document}
%Author 1

\author[D.~Corro]{Diego Corro$^{\ast\dagger}$}
\address[D.~CORRO]{Institut f\"ur Algebra und Geometrie, Karlsruher Institut f\"ur Technologie (KIT), Karlsruhe, Germany.}
\curraddr{Instituto de Matemáticas, sede Oaxaca, Universidad Nacional Autó\-noma de México (UNAM), Mexico.}
%\email{\href{mailto:diego.corro@partner.kit.edu}
%{diego.corro@partner.kit.edu}}
\email{\href{mailto:diego.corro.math@gmail.com}
{diego.corro.math@gmail.com}}
\email{\href{mailto:diego.corro@im.unam.mx}
{diego.corro@im.unam.mx}}

\author[K.~Garcia]{Karla Garcia$^\ast$}
\address[K.~GARCIA]{Institut f\"ur Algebra und Geometrie, Karlsruher Institut f\"ur Technologie (KIT), Karlsruhe, Germany.}
\curraddr{Facultad de Ciencias, Universidad Nacional Autónoma de México (UNAM), Mexico.}
%\email{\href{mailto:diego.corro@partner.kit.edu}
%{diego.corro@partner.kit.edu}}
\email{\href{mailto:ohmu@ciencias.unam.mx}
{ohmu@ciencias.unam.mx}}

\author[M.~G\"{u}nther]{Martin~G\"{u}nther$^\ast$}
\address[M.~G\"{U}NTHER]{Institut f\"ur Algebra und Geometrie, Karlsruher Institut f\"ur Technologie (KIT), Karlsruhe, Germany.}
%\email{\href{mailto:diego.corro@partner.kit.edu}
%{diego.corro@partner.kit.edu}}
\email{\href{mailto:martin.guenther@kit.edu}
{martin.guenther@kit.edu}}

\author[J.-B.~Korda\ss]{Jan-Bernhard~Korda\ss$^\ast$}
\address[J.-B.~KORDA\ss]{Département de Mathématiques, Université de Fribourg, Fribourg, Switzerland.}
%\email{\href{mailto:diego.corro@partner.kit.edu}
%{diego.corro@partner.kit.edu}}
\email{\href{mailto:jan-bernhard.kordass@unifr.ch}
{jan-bernhard.kordass@unifr.ch}}

\thanks{$^{\ast}$ Supported by the DFG (281869850, RTG 2229 ``Asymptotic Invariants and Limits of Groups and Spaces'').}

\thanks{$^{\dagger}$ Supported by a DGAPA Postdoctoral Fellowship of the Institute of Mathematics, UNAM}

%-----------------------------------------------------------------------------------

\title[Even-dim. spher. space form bundles and 1/4-pinched Riem. metrics
%1/4-pinched even dim. spherical space forms bund. %and Ricci flow
]{Bundles with even-dimensional spherical space form as fibers and fiberwise quarter pinched Riemannian metrics
%Quarter pinched even dimensional spherical space forms bundles %and Ricci flow
}
\date{\today}

% MATH SUBJECT CLASSIFICATION AND KEYWORDS

\subjclass[2010]{57R22, 53C10}
\keywords{$1/4$-pinched metrics, spherical space form bundle}

%-----------------------------------------------------------------------------------
\setlength{\overfullrule}{5pt}
	\begin{abstract}
%We show that for any smooth bundle with fiber a projective space, such that every fiber carries in a continuous way a pointwise $1/4$-pinched metric, the structure group reduces to the isometry group of the standard round metric.
Let $E$ be a smooth bundle with fiber an $n$-dimensional real projective space $\mathbb{R}P^n$. We show that, if every fiber carries a positively curved pointwise strongly $1/4$-pinched Riemannian metric that varies continuously with respect to its base point, then the structure group of the bundle reduces to the isometry group of the standard round metric on $\mathbb{R}P^n$.
	\end{abstract}

%---------------------------------------------------------------------------------
\maketitle

% WHAT AND WHY?

%\todo[inline]{$1/4$ vs quarter pinched (maybe just decide on one)}\todo[inline]{D: put everything in $1/4$-pinched form}

\section{Main results}

% WHAT IS THE TOPIC?
% AND WHY IS IT INETERESTING/RELEVANT?
% HOW DO YOU APPROACH/SOLVE/OUTLINE?

Recently, the Generalized Smale Conjecture has been proven for $M$ %a manifold $M$ when $M$ is  the $3$-sphere, or any $3$-dimensional spherical space form ;
the $3$-dimensional sphere, or any $3$-dimensional spherical space form; that is, the diffeomorphism group of $M$ is homotopy equivalent to the isometry group of a metric of constant sectional curvature (see \cite{BamlerKleiner2017,BamlerKleiner2019,
  Hatcher1983,Smale1959}). Farrell, Gang, Knopf and Ontaneda in \cite{FarrellGangKnopfOntaneda2017} proved using Ricci flow that a similar result holds for the structure group of a smooth sphere bundle over a compact manifold, when each fiber has a $1/4$-pinched Riemannian metric which depends continuously on the base point, i.e.\ a \emph{fiberwise metric}. %\todo{using Ricci flow}
Henceforth, by a $1/4$-pinched metric we will always refer to a metric with \emph{positive} sectional curvatures bounded between $\frac{1}{4}$ and $1$.

In the present work, we study the more general case of  smooth fiber bundles with  an arbitrary spherical space as fiber, equipped with a  \emph{fiberwise pointwise strongly  $1/4$-pinched Riemannian metric}.

By a \emph{smooth bundle} over a space $B$, we mean a locally trivial bundle over $B$ whose structural group is $\Diff(F)$, the group of self-diffeomorphisms of $F$ with the smooth topology. We assume that each fiber is equipped with a Riemannian metric such that at any point of the fiber,  the ratio of the maximal to the minimal sectional curvatures is strictly less than $4$

We point out that for any  $n$-dimensional spherical space form $F$, the standard round metric $\widetilde{\sigma}$ of the unit round sphere $S^n\subset \R^{n+1}$  induces a distinguished metric $\sigma$ on $F$, which we call the \emph{standard round metric} of $F$.
% (in this note we will only consider the round unit $n$-dimensional sphere)\todo{remove sentence in parantheses?}.
In the case when the bundle has fibers diffeomorphic to a real projective space and has a fiberwise pointwise strongly $1/4$-pinched Riemannian metric, we conclude that the structure group of the bundle reduces:

\begin{theorem}\label{Theorem: Main theorem}
Let $E\to B$ be a smooth fiber bundle over a locally compact space $B$, whose fibers are real projective spaces. %Denote by $I$ the isometry group of the standard round metric.
If the bundle admits  pointwise strongly $1/4$-pinched fiberwise metrics, then its structure group reduces to the isometry group of the standard round metric.%$I$.
\end{theorem}

Recall that an even dimensional  spherical space form is either diffeomorphic to the sphere or a real projective space $\mathbb{R}P^n$ (cf.\ \cite[Proposition 4.4, p.166]{doCarmo}. Thus, for the particular case of even dimensional fibers, % projective spaces,
we obtain the following corollary, which is a direct generalization of the conclusions of \cite{FarrellGangKnopfOntaneda2017}:

\begin{corollary}\label{Theorem: even dimension}
%Let $E\to B$ be a smooth fiber bundle with fibers a projective space. Denote by $I$ the isometry group of the standard round metric. If the bundle admits $1/4$-pinched fiberwise metrics, then the structure group reduces to $I$.
Let $E\to B$ be a smooth fiber bundle over a locally compact space $B$, with fibers homeomorphic to an even dimensional spherical space form $F$. %Denote by $I$ the isometry group of the standard round metric of $F$.
If the bundle admits pointwise strongly  $1/4$-pinched fiberwise metrics, then the structure group reduces to the isometry group of the standard round metric.%$I$.
\end{corollary}

% \todo{Add Theorem saying that we can do it in general for projective spaces}

%, so Theorem~\ref{Theorem: Main theorem} is a complete generalization of the Main Theorem in \cite{FarrellGangKnopfOntaneda2017} when the fiber is of even dimension.
For the particular case of  dimension $3$,  the work of  Bamler and Kleiner \cite{BamlerKleiner2017,BamlerKleiner2019}  implies that a reduction of the structure group as in Corollary~\ref{Theorem: even dimension} holds for any smooth bundle with fiber a $3$-dimensional spherical space form, independently of whether the bundle carries a $1/4$-pinched fiberwise metric. This same observation holds for arbitrary $\R P^2$-bundles, see \cite[Proof of Theorem~3.4, p. 364]{RosenbergStolz}.

%The original result by Smale in \cite{Smale1959} gives a strong deformation retract of $\Diff (S^2)$ onto $\Iso (S^2)$. By the proof of Lemma 2.2 in \cite{BamlerKleiner2017}, this is equivalent to the fact that the space of smooth Riemannian metrics of constant curvature $1$, denoted by $\Met^1(S^2)$, is contractible with the topology induced by the inclusion $\Met^1(S^2)\subset \Met(S^2)$\todo{which is equipped with the smooth whitney top}. We also consider the space of Riemannian metrics of constant sectional curvature $1$ on $\R P^2$, denoted by $\Met^1 (\R P^2)$ with the induced topology of $\Met(\R P^2)$. %Fixing a point in $\R P^2$ and a preimage of it in $S^2$ via the covering $S^2\to \R P^2$,
%We have a continuous map $\Met^1(\R P^2)\to \Met^1 (S^2)$ given by lifting the metric on $\R P^2$ via the covering $S^2\to \R P^2$. Composing with the deformation retraction $\widetilde{F}\colon \Met^1(S^2)\times I \to \Met^1(S^2)$, we get a  homotopy equivalence of $\Met^1(\R P^2)$ to the round metrics of $S^2$. By the fact that any metric in $\Met^1(S^2)$ is isometric to the standard round metric, it follows that any metric in $\Met^1(S^2)$ is invariant under the antipodal map, and thus we can project them onto $\R P^2$. In this fashion we get a deformation retract for $\Met^1(\R P^2)$. Thus the Generalized Smale Conjecture also holds for $\R P^2$\todo{Add ref to \cite[Proof of Theorem 3.4]{RosenbergStolz}, say what generalized smale conjecture is}.

For bundles with spherical space form as fibers of odd dimension we give some sufficient conditions to obtain the same conclusions of Theorem~\ref{Theorem: Main theorem}. In \cite{FarrellGangKnopfOntaneda2017}, Farrell, Gang, Knopf and Ontaneda show that for any trivialization $\alpha_i\colon p^{-1}(U_i)\to U_i\times F$ of a smooth fiber bundle $S^n\to \tilde{E} \overset{p}{\to} B$ equipped with a fiberwise metric $\{ g_b \}_{b\in B} $ of constant sectional curvature $1$ there is a continuous map $\phi_i\colon U_i \to \Diff(S^n)$, which for each $b \in U_i$ takes the pushforward metric $(\alpha_{i})_* g_b$ on $S^n$ and then sends $b$ to an isometry $\phi_i (b) : (S^n, (\alpha_{i})_* g_b) \rightarrow (S^n, \widetilde{\sigma} ) $. Denote by $\Gamma$ any finite subgroup of $\mathrm{O}(n+1)$ acting freely by isometries on $S^n$ with respect to the standard round metric, and  let $F$ be the spherical space form  $S^n/\Gamma$. Given a smooth fiber bundle $p\colon E\to B$ with fiber $F$, we say that an $n$-sphere bundle  $\widetilde{p}\colon\widetilde{E}\to B$ is a \emph{covering bundle} of $p\colon E\to B$, if there exist a map $\pi\colon \widetilde{E}\to E$ which is an extension of the covering $S^n\to F$ of the fibers. Observe that in general $\widetilde{E}$ is not the universal covering of $E$, since for any sphere bundle over $B$ the fundamental group of the total space is isomorphic to the fundamental space of the base.

In the next theorem we give a sufficient condition for the structure group of a smooth fiber bundle, with fibers an $n$-dimensional space form to reduce to the isometry group of the standard round metric on the fiber.

%In the next theorem, for a smooth fiber bundle with a fibers $F$ an spherical space form of arbitrary dimension,  we give sufficient conditions to imply the reduction of its structure group.

\begin{theorem}\label{Theorem: Covering}
Let $F$ be an $n$-dimensional spherical space form with fundamental group $\Gamma$. Denote by $I= \mathrm{Isom}(F,\sigma)$ the group of isometries of $F$ with respect to the standard round metric $\sigma$. Let $F\to E\overset{p}{\to} B$ be a smooth fiber bundle over a locally compact space $B$, with pointwise strongly quarter pinched fiberwise metrics.
Assume that there exist a covering bundle $S^n\to \tilde{E}\overset{\tilde{p}}{\to} B$ extending  $\pi \colon S^n\to F$. Consider $\phi_i\colon U_i\to \Diff(S^n)$ the continuous maps described above. When the images $\phi_i(U_i)$ are contained in the normalizer of $\Gamma$ in $\Diff(S^n)$,  i.e.\ $\phi_i(U_i)\subset N_{\Diff(S^n)}(\Gamma)$, for all indices $i$, then %\todo{is contain in the normalizer and then notation},
the structure group of $p\colon E\to B$ reduces to $I$.
\end{theorem}

We point out that in general, for an arbitrary smooth fiber bundle $F\to E\to B$ with spherical space form as fibers, an extension $\tilde{E}\to E$ of the covering $S^n\to F$ may not exist.  In Section~\ref{S: spherical space forms} we exhibit  a family of  smooth $\R P^3$-bundles over closed simply-connected $4$-manifolds which do not admit an extension of the covering $S^3\to \R P^3$.

The proofs of both Theorems~\ref{Theorem: Main theorem} and \ref{Theorem: Covering}  rely on the fact that, locally,  the spherical space form fiber bundle under consideration is finitely covered by a trivial sphere bundle. With this observation, several results presented in \cite{FarrellGangKnopfOntaneda2017} can be generalized. The main difference between Theorem~\ref{Theorem: Main theorem} and Theorem~\ref{Theorem: Covering} is  an equivariance problem. Namely, when the fibers are projective spaces, we can do all constructions in a $\Z_2$ equivariant fashion. It is not clear whether the analogous constructions can be done in a continuous fashion for an arbitrary finite group $\Gamma$ of $\mathrm{O}(n+1)$ acting freely and by isometries on the unit round  $n$-sphere, and thus we impose the additional, sufficient conditions in Theorem~\ref{Theorem: Covering}.

The present article is organized as follows: First we give the preliminaries needed for our discussion and recall several facts from \cite{FarrellGangKnopfOntaneda2017}, for the sake of completeness. In Section~\ref{S: spherical space forms}, we prove Theorem~\ref{Theorem: Covering}, and in the last section we prove Theorem~\ref{Theorem: Main theorem}.

	% ACKNOWLEDGMENTS
%----------------------------------------------------------------------------------------------------

\begin{ack} We thank Fernando Galaz-Garcia, for comments on the first versions of the present manuscript, and Wilderich \linebreak Tuschmann for useful conversations.
\end{ack}
%----------------------------------------------------------------------------------------------------
	% Preliminaries
%----------------------------------------------------------------------------------------------------

\section{Preliminaries}\label{S: Preliminaries}
%
%\\
%----------------------------------------------------------------------------------------------------

For a closed smooth manifold $M$, we denote by $\Met(M)$ the space of smooth Riemannian metrics on $M$ equipped with the \emph{$C^{\infty}$ Whitney topology}. % (also known as the \emph{strong topology}).
For a survey on the properties of this space, the interested reader can consult, for example, \cite{CorroKordass2019}, \cite{Hirsch}, \cite{KrieglMichor}, \cite{TuschmannWraith}.
A Riemannian metric on $M$ is \emph{pointwise strongly $1/4$-pinched} if for any point $p\in M$, the ratio of the maximal to the minimal sectional curvatures at $p$ is strictly less than $4$.
We denote by $\Met^{1/4}(M)$ the space of pointwise strongly $1/4$-pinched metrics with the subspace topology induced by $\Met^{1/4}(M)\subseteq \Met(M)$.
In an analogous fashion we define $\Met^{1}(M)$ to be the \emph{space of round metrics}, i.e.\ metrics with constant sectional curvature $1$, with the induced topology given by $\Met^1(M)\subset\Met(M)$.

Let $F\to E \overset{p}{\to} M$ be a smooth fiber bundle with compact fiber $F$.
A \emph{fiberwise metric on the bundle} is a continuous family of Riemannian metrics $g_x$ on every fiber $F_x$. Here, by ``continuous'' we mean that for any local trivialization $\alpha\colon  p^{-1}(U)\to U \times F$ of $p\colon E\to B$,  the induced map $U \to \Met(F)$ given by $x\mapsto \alpha_{\ast} g_x$ is continuous. Observe that for two local trivalizations, $\alpha_i\colon p^{-1}(U_i)\to U_i\times F$ and $\alpha_j\colon p^{-1}(U_j)\to U_i\times F$, the transition map $\alpha_{ij}(b) = \alpha_j\circ \alpha_{i}^{-1}\colon (U_i\cap U_j) \times F\to (U_i\cap U_j)\times F$ is of the form
\[
	\alpha_{ij}(b,v) = (b,\alpha_{ij}(b)(v)).
\]
The definition of fiberwise Riemannian metric implies that for any $b\in U_i\cap U_j$ the map $\alpha_{ij}(b) \colon (F,(\alpha_i)_\ast \, g_b)\to (F,(\alpha_j)_\ast \, g_b)$ is an isometry.

The \emph{normalized Ricci flow} on a smooth compact $n$-dimensional manifold $M$ is the solution of the following differential equation:
\[
\frac{\partial}{\partial t}\, g(t) = -2 \Ric_{g(t)}+\frac{2}{n}\left(\frac{\int_M \mathrm{Scal}(g(t))d\mathrm{vol}(g(t))}{\int_M d\mathrm{vol}(g(t))}\right)g(t),
\]
where $g(t)$ is a time-dependent family of Riemannian metrics on $M$. In  \cite{BrendleSchoen2009}, Brendle and Schoen proved that for a closed manifold $M$ of dimension $n\geqslant 4$, the normalized Ricci flow starting at a pointwise strongly $1/4$-pinched metric converges to a metric with constant sectional curvature $1$. In dimension $3$ this same result holds by the work of Hamilton in \cite{Hamilton1982}.

Consider a finite subgroup $\Gamma < \mathrm{O}(n+1)$ acting freely on the unit round sphere. %$S^n\subset\R^{n+1}$.
A \emph{spherical space form} is the quotient manifold $S^n/\Gamma$. Since $\Gamma$ is a subgroup of isometries of the standard round metric $\tilde{\sigma}$ of $S^n$, the standard round metric on $S^n$ induces  a distinguished round metric $\sigma$ on $S^n/\Gamma$. We refer to this metric as the \emph{standard round metric on $S^n/\Gamma$}.

As in the case of sphere bundles, for a smooth fiber bundle with fiber a spherical space form, if the fibers are pointwise strongly $1/4$-pinched, then the normalized Ricci flow gives a continuous fiberwise deformation of the fiberwise metrics to a round metric (not necessarily the standard one). This leads to the following fiberwise version of \cite[Theorem~1]{BrendleSchoen2009}:
%For the particular case   of spherical spaceforms, i.e.\ $S^n/\Gamma$ where $\Gamma\subseteq \Isom(S^n,\sigma)$ is a discrete subgroup acting freely by isometries of the standard round metric $\sigma$.

\begin{prop}\label{P: pointwise stron 1/4-pinched flows to round metric}
Let $F=S^n/\Gamma$ be an $n$-dimensional spherical space form, and $\pi\colon E\to B$ a smooth $F$-fiber bundle over a locally compact space $B$. Assume that the bundle admits strongly $1/4$-pinched fiberwise metrics, then the bundle admits fiberwise round metrics.
%\begin{enumerate}[(i)]
%\item the fiberwise Ricci flow is continuous as we move in the base space;
%\item
%\end{enumerate}

\end{prop}

\begin{proof}
%Fix $b\in B$ and consider $F_b =\pi^{-1}(b)$ the fiber over $b$. Denote by $g_b$ the Riemannian metric on the fiber. For the universal cover $p_b\colon(S^n)\colon F_b$, there exists a Riemannian metric $\tilde{g}_b$ on $S^n$ for which $p_b\colon (S^n,\tilde{g}_b)\to (F_b,g_b)$ is a local isometry. This implies that $\tilde{g}_p$ is also point-wise strongly $1/4$ pinched. Denote by $\tilde{g}(t)_b$ the solution of  the normalized Ricci flow starting at $\tilde{g}_b$, and  $\tilde{g}^\bullet_b$ the limit of  the flow. Recall that $\tilde{g}^\bullet(t)_b$ is a round metric. Since $p_b\colon (S^n,\tilde{g}_b)\to (F_b,g_b)$ is a local isometry, and a principal $\Gamma$-bundle, then $\Gamma$ acts by isometries on $S^n$ with respect to $\tilde{g}_b$. Since the normalized Ricci flow preserves isometries \todo{Add reference}, we have that $\Gamma$ acts by isometries with respect to $\tilde{g}(t)_b$ for any time $t$. Thus we obtain a family of Riemannian metrics $g(t)_b$ on $F_b$, which are continuous with respect to $t$ and have a limit $g^\bullet_b$, which is a round metric on $F_b$.
Locally, the continuous fiberwise family of Riemannian metrics of $\pi\colon E\to B$ is given by a continuous map $F\colon U\to \Met^{1/4}(S^n/\Gamma)$, where 
$U\subset B$ is a open set over which the fiber bundle is trivial.
Since $B$ is locally compact, we can assume w.l.o.g.\ that $\bar U$ is compact and the bundle is also trivial over $\bar U$.
% and, since $B$ is locally compact, we can assume that $\bar{U}$ is compact. %\todo{Martin: I think we need a local trivialisation here. Otherwise we have to identify all the fibers with $S^n/\Gamma$ in a continuous way, which makes the bundle trivial.}.
% By lifting the metric $F(b)$ on $S^n/\Gamma$ via the universal cover $p\colon S^n\to S^n/\Gamma$, we get $\widetilde{F}(b)$ on $S^n$ \todo{this is redundant!}
Via the universal cover $p\colon S^n\to S^n/\Gamma$ we obtain a continuous map $\widetilde{F}\colon \bar U\to \Met^{1/4}(S^n)$. The metric $\widetilde{F}(b)$ on $S^n$ is the pull-back of $F(b)$ along $p$. Consider the map $F^\bullet\colon \bar U\to \Met^1(S^n)$ that sends $\widetilde{F}(b)$ to its limit $\widetilde{F}^\bullet(b)$ under the normalized Ricci flow.  Theorem~1 in \cite{FarrellGangKnopfOntaneda2017} implies that this map is continuous on the compact set $\bar U$, hence on $U$. Furthermore, for each fixed $b\in B$, by construction the action of $\Gamma<\mathrm{O}(n+1)$ on $S^n$ is by isometries with respect to the metric $\widetilde{F}(b)$. Since the normalized Ricci flow preserves isometries (see \cite{AndrewsHopper}), %\todo{citation. DC: can not find it explicitly, I think it is easy to check}\todo{Karla: It is in the book by Andrews and Hopper}
this implies that the action of $\Gamma$ on $S^n$ is by isometries with respect to $\widetilde{F}^\bullet(b)$. Thus we obtain a new Riemannian metric $F^\bullet(b)$ on $S^n/\Gamma$.

Recall that two metrics in $\Met(S^n/\Gamma)$ are close in the Whitney topology if for any atlas, the matrix coefficients are uniformly close. Since $\widetilde{F}^\bullet $ depends continuously on $b$, and $p\colon (S^n,\widetilde{F}^\bullet(b))\to (S^n/\Gamma,F^\bullet(b))$ is a local isometry, this implies that $F^\bullet\colon \bar U\to \Met^1(S^n/\Gamma)$ is a continuous function.

%with using a local cover argument and Theorem~1 in \cite{FarrellGangKnopfOntaneda2017}, it can be shown that the limit of the normalized Ricci flow on a s on a fiber bundle with spherical space form fibers is a continuous fiberwise round metric.
\end{proof}

%\todo{Missing link}

For the sake of completeness, we recall  \cite[Lemma~1]{FarrellGangKnopfOntaneda2017}, which will be used in Section~\ref{S: spherical space forms}.

\begin{lem}[Lemma~1 in \cite{FarrellGangKnopfOntaneda2017}]
\label{L: every round metric is isometric to round in a continuous way}
Let $\widetilde{\sigma}$ denote the standard round metric on $S^n$. There is a continuous map
\[
	\phi\colon \Met^{1}(S^n)\to \Diff(S^n),
\]
such that $\phi(g)\colon (S^n,g)\to (S^n,\widetilde{\sigma})$ is an isometry, for any round metric $g\in \Met^{1}(S^n)$.
\end{lem}

We observe that, for a round metric $g$ with isometry group $\Iso(g)$, the map $\phi(g)$ is in general not $\Iso(g)$-equivariant. Given any finite subgroup $\Gamma$ of $\mathrm{O}(n+1)$, acting freely and isometrically on the unit round sphere, and $f\in \Diff(S^n)$, we obtain a new effective representation  $\rho_f\colon \Gamma \hookrightarrow \Diff(S^n)$, defined as follows: For any $\gamma\in \Gamma$ and $v\in S^n$ we set
%%We observe that, for any round metric $g$ with isometry group $\Gamma$, the map $\phi(g)$ is not $\Gamma$-equivariant.
%In general, given any finite subgroup $\Gamma$ of $\mathrm{O}(n+1)$, acting freely on $S^n$, and $f\in \Diff(S^n)$, we obtain a new effective representation  $\rho_f\colon \Gamma \hookrightarrow \mathrm{O}(n+1)$. This representation is defined as follows: For any $\gamma\in \Gamma$ and $v\in S^n$ we have
\[
	\rho_f(\gamma)(v) = f\big(\gamma\big(f^{-1}(v)\big)\big).
\]

\begin{remark}
In the case that such representation $\rho_f$ is again a subgroup of $\mathrm{O}(n+1)$, de~Rahm showed in \cite{deRahm1950} that in odd dimensions $\rho_f$ has the same character as the inclusion $\Gamma\hookrightarrow \mathrm{O}(n+1)$. This implies that there is some $h_f\in \mathrm{O}(n+1)$,  such that for all $\gamma \in \Gamma$   we have
\[
	h_f^{-1}\gamma h_f = \rho_f(\gamma).
\]
%\todo{does this hold elementwise?}\todo{Use pointwise notation (Diego/Martin)?!}
In other words, the map $h_f^{-1}\circ f\in \Diff(S^n)$ is $\Gamma$-equivariant. When we consider the maps $h_{\phi(g)}^{-1}\circ \phi(g)$, we get an equivariant map. It is not clear, however, whether this map is continuous with respect to the metric $g$.
% but this new map may not be continuous with respect to the metric $g$.
\end{remark}

 \section{General spherical space form case}\label{S: spherical space forms}
%
%%%%%%%%%%%%%%%%%%%%%%%%%%%%%%%%

Let $p\colon E\to B$ and $\hat p\colon \hat E\to B$ be two smooth fiber bundles, $\hat F$ and $F$ their fibers over a specific base point $b\in B$, and $\pi\colon \hat F \to F$ a covering.\\
We call a covering $\pi\colon \hat E \to E$ a \emph{covering of the bundle $E\to B$ extending the covering $\pi\colon \hat F \to F$} if the following diagram commutes:

%%%%%%%%%%%%%%%%%%%%%
%%%% old version %%%%
%Given a smooth fiber bundle $p\colon E\to B$ with fiber $F$ and  a covering $\pi\colon \hat F\to F$, %\todo{trying to avoid confusion with the universal cover $\tilde{F}\to F$}
%we call a smooth fiber bundle $\hat{p}\colon \hat{E}\to B$ a \emph{covering of the bundle $E\to B$ extending the covering $\hat {F}\to F$}, if  $\pi \colon \hat{E}\to E$ is a cover  extending $\pi\colon \hat F\to F$, such that $\hat{p}(\hat{e}) = p\circ \pi (\hat{e})$. This means that there exists some $b\in B$ such that for the fibers $p^{-1}(b) = F_b$ and $\hat{p}^{-1}(b) = \hat F_b$, the following diagram commutes: %\todo{same notation with widetilde}
%
%
%d a covering $\pi \colon \hat E\to E$ such that the fibers of the fibration  $p\circ \pi \colon \hat E\to B$ are homotopy equivalent to $\hat F$, i.e.\ the following diagram commutes:
%
%%%%%%%%%%%%%%%%%%%
\begin{equation}\label{E: Comm. Diagram covering of fibrations}
\begin{tikzcd}
\hat F \arrow{r}{i} \arrow{d}{\pi} & \hat{E} \arrow[bend left]{dr}{\hat{p}} \arrow{d}{\pi} & \\
F \arrow{r}{i} & E \arrow{r}{p} & B
\end{tikzcd}
\end{equation}

\begin{remark}
We will only consider  the universal cover $\pi \colon \widetilde{F}\to F$, and given a fiber bundle $p\colon E\to B$ with fiber $F$, we denote a covering bundle of $\widetilde{F}\to F$ by $\widetilde{p}\colon \widetilde{E}\to B$. Note that $\widetilde{E}$ is, in general, not the universal cover of $E$, since the fundamental group of $\widetilde{E}$ is isomorphic to the fundamental group of $B$.
\end{remark}

% \todo{$=S^n/\Gamma$?}{}
Let $F$ be an $n$-dimensional spherical space form. Denote by $I$ the isometry group  $\mathrm{Isom}(F,\sigma)$ of $F$ with respect to the standard round metric $\sigma$, and by $G$ the structure group of a smooth fiber bundle $F\to E\overset{p}{\to} B$. For the universal cover $\pi\colon S^n\to F$, let $G^\ast$ denote the group of diffeomorphisms of $S^n$ which  consists of all the lifts of elements in $G$ .
Observe that there is a natural continuous group homomorphism $\psi\colon G^\ast\to G$.

%\begin{equation*}
%\begin{tikzcd}
%\hat F \arrow{r} \arrow{d}{\pi} & \hat E \arrow[bend left]{dr}{\hat p} \arrow{d}{\pi} & \\
%F \arrow{r}{i} & E \arrow{r}{p} & B
%\end{tikzcd}
%\end{equation*}
%
%We call such a map $\hat E \to E$ a \emph{covering of the fibration $E\to B$ extending the covering $\hat F\to F$}.

The bundle $p\colon E\to B$ admits a smooth covering fiber bundle $\widetilde{p}\colon \widetilde{E}\to B$ extending the universal cover $\pi \colon S^n \to F$ if there exists a continuous group homomorphism $c\colon G\to G^\ast$ such that $\psi\circ c = Id_G$ (see \cite[Section~3]{BeckerGottlieb1973}). In particular, if $p\colon E\to B$ admits a pointwise strongly $1/4$-pinched fiberwise metric, so does $\widetilde{p}\colon\widetilde{E} \to B$. From the Main Theorem in \cite{FarrellGangKnopfOntaneda2017}, the structure group of $\widetilde{p}\colon \widetilde{E}\to B$ is $\mathrm{O}(n+1)$. We give sufficient conditions to guarantee that this reduction descends to $p\colon E\to B$.

\begin{remark}\label{R: how to make things equivariant}
Recall that for a local trivialization $\widetilde{\alpha}_i\colon\widetilde{p}^{-1} (U_i) \to U_i\times S^n$ of the bundle $\widetilde p\colon \widetilde E\to B$ with round fiberwise metrics, the map $\phi$ given by Lemma~\ref{L: every round metric is isometric to round in a continuous way} induces a map $\phi_i\colon U_i\to \Diff(S^n)$ as follows:
\[
	\phi_i(b) = \phi((\widetilde{\alpha}_i)_\ast\,g_b),
\]
where $g_b$ is the metric on the fiber $\widetilde{p}^{-1}(b)$.
\end{remark}

\begin{lem}\label{L: covering bundle with fiberwise round metrics}
Let $F$ be an $n$-dimensional spherical space form with fundamental group $\Gamma$. %Denote by $I$ the group of isometries of $F$ with respect to the standard round metric $\sigma$.
Let $F\to E\overset{p}{\to} B$ be a smooth fiber bundle over a locally compact space $B$, with fiberwise round metrics.
Assume that there exists a covering bundle $S^n\to \widetilde{E}\overset{\widetilde{p}}{\to} B$ extending  $\pi \colon S^n\to F$. Consider  the continuous maps $\phi_i\colon U_i\to \Diff(S^n)$ from Remark~\ref{R: how to make things equivariant}. %induced by Lemma~\ref{L: every round metric is isometric to round in a continuous way} described above.
If $\phi_i(U_i)\subset N_{\Diff(S^n)}(\Gamma)$ for all indices $i$, then the structure group of $p\colon E\to B$ reduces to the isometry group of $F$ with respect to the standard round metric.
\end{lem}

\begin{proof}
Since $p\colon E\to B$ has a family of fiberwise round metrics Riemannian metrics, so does $\widetilde{p}\colon \widetilde{E}\to B$. By \cite{FarrellGangKnopfOntaneda2017}, the structure group of $\widetilde{p}\colon \widetilde{E} \to B$ reduces to $\mathrm{O}(n+1)$. We recall how such a reduction is obtained. Fix  a local trivalization  $\{(U_i,\alpha_i)\}$ of $p\colon E\to B$. Denote by $\widetilde{\sigma}$ the standard round metric on $S^n$, and consider  the trivialization $\widetilde{\alpha}_i\colon \widetilde{p}^{-1}(U_i)\to  U_i\times S^n$ induced by $\alpha_i\colon p^{-1}(U_i)\to U_i\times F$. Fix $x\in U_i\cap U_j$ and consider
\[
	\phi_i(x) \colon \left(S^n,(\widetilde{\alpha}_i)_\ast\,\widetilde{g}_x \right) \to \left(S^n,\widetilde{\sigma} \right).
\]
This is the isometry described in Lemma~\ref{L: every round metric is isometric to round in a continuous way}. Since $\phi_i(x)$ and $\alpha_{ij}(x)$ are isometries, then by definition the map $\widetilde{\beta}_{ij}\colon U_i\cap U_j\to \Diff(S^n)$ given by
\[
	\widetilde{\beta}_{ij}(x) = \phi_j(x)\circ(\widetilde{\alpha}_{ij}(x))\circ \phi_i(x)^{-1},
\]
is an element of $\mathrm{O}(n+1)$. Observe that the hypothesis $\phi_i(x)\in N_{\Diff(S^n)}(\Gamma)$ implies that both maps  $\phi_i(x)$ and $\phi_j(x)$ leave the orbits invariant, i.e.\ a point $q$ in the orbit $\Gamma(p)$ is mapped to another point in $\Gamma(p)$. This is also true for $\widetilde{\alpha}_{ij}(x)$ by construction. Thus, $\widetilde{\beta}_{ij}(x)$ maps orbits to orbits. This implies that for any $x\in U_i\cap U_j$, the map $\widetilde{\beta}_{ij}(x)$ induces  an isometry $\beta_{ij}(x) \colon \left(F,\sigma\right)\to (F,\sigma)$, i.e.\ we have a reduction to the isometry group of $(F,\sigma)$.
\end{proof}

\begin{proof}[Proof of Theorem~\ref{Theorem: Covering}]
By Proposition~\ref{P: pointwise stron 1/4-pinched flows to round metric}, the fiber bundle $E\to B$ admits round fiberwise metrics. The result follows from Lemma~\ref{L: covering bundle with fiberwise round metrics}.
\end{proof}

\subsection{Obstructions to lifting spherical space form bundles to sphere bundles}\hspace*{\fill}\\

Given a fibration $p\colon E\to B$ with fiber $F = p^{-1}(b)$, for some fixed $b\in B$, and a covering $\pi\colon \hat E\to E$, the map $\hat{p} \coloneqq p\circ \pi \colon \hat E \to B$ is also a fibration, with fiber $\hat F = \hat p^{-1}(b)$.
If $\hat F$ is connected, the restriction $\pi\colon \hat F \to F$ is then also a covering, and we call $\pi\colon \hat E\to E$ a \emph{covering of the fibration $E\to B$ extending the covering $\hat F\to F$}. In this case, diagram \eqref{E: Comm. Diagram covering of fibrations} commutes.

A smooth covering fiber bundle is a particular example of a covering fibration, but in general not every covering fibration is a smooth fiber bundle. The following theorem gives sufficient and necessary conditions for the existence of a covering fibration for the universal cover $\pi\colon \widetilde{F}\to F$.

\begin{thm}[Theorem~1 in \cite{BeckerGottlieb1973}]
\label{T:Obstructions to Covering fibrations}

Let $p\colon E\to B$ be a fibration with fiber $F$ and $\pi\colon \widetilde{F}\to F$ be the universal cover. Then, a covering $\widetilde{E} \to E$ of the fibration $E\to B$ extending $\widetilde{F}\to F$
exists if and only if  the following two conditions are satisfied:
\begin{enumerate}
\item $i_\ast\colon \pi_1(F)\to \pi_1(E)$ is injective and,
\item $p_\ast\colon \pi_1(E) \to \pi_1(B)$ has a right inverse homomorphism.
\end{enumerate}
\end{thm}

With this characterization, we show that there are smooth fiber bundles with spherical space form as fibers which do not %admit even a covering fibration extending the universal covering, and thus, these smooth bundles do not
satisfy the hypothesis of Theorem~\ref{Theorem: Covering} relating to the existence of a covering bundle $S^n\to \tilde{E}\to B$.

\begin{ex}\label{Ex: Example1}
Consider $\SO(3) \to E \SO(3) \to B \SO(3)$  the universal bundle of $\SO(3) \approx \R P^3$. From the long exact sequence of the universal bundle, we see that $\pi_{k+1}(B\SO(3))\cong \pi_k(\SO(3))$, and in particular $\pi_3(B\SO(3))=\pi_2(\SO(3))=0$ and $\pi_2(B\SO(3))=\pi_1(\SO(3))=\Z_2$. Since $\pi_1(E\SO(3))$ is the trivial group, then the  map $i^\ast\colon \pi_1(\SO(3))=\Z_2\to \pi_1(E\SO(3))$ is not injective. Thus by Theorem~\ref{T:Obstructions to Covering fibrations} there does not exist a covering fibration extending the universal cover $S^3\to \R P^3 \approx \SO(3)$.
\end{ex}

From this previous example we can give more examples of spherical space form bundles which do not admit a covering fibration.

\begin{ex}\label{Ex: Example}
For any simply-connected smooth $4$-dimensional manifold $M$, with $H_2(M,\Z) = \Z^ n$, we  construct a smooth fiber bundle $\R P^3 \to E \to M$ which does not have a covering fibration $S^3\to \widetilde{E} \to M$ extending the universal cover $S^3\to \R P^3$, and hence it does not admit a covering bundle.

%Let $\SO(3) \to E \SO(3) \to B \SO(3)$ be the universal bundle of $\SO(3) \approx \R P^3$. From the long exact sequence of the universal bundle, we see that $\pi_{k+1}(B\SO(3))\cong \pi_k(\SO(3))$, and in particular $\pi_3(B\SO(3))=\pi_2(\SO(3))=0$ and $\pi_2(B\SO(3))=\pi_1(\SO(3))=\Z_2$.
Recall from  the proof of Theorem 1.2.25 in \cite{GompfStipsicz1999}, %\todo{check source}
that every simply-connec\-ted smooth $4$-manifold $M$ with $H_2(M,\Z) \cong \Z^n$ is homotopy equivalent to a CW-complex of the form
\[
M \simeq \bigg(\bigvee_{i=1}^n S^2\bigg) \cup_{S^3} D^4,
\]
where the boundary of $D^4$ is attached to the wedge of spheres via a gluing map $F\colon S^3 \to \bigvee_{i=1}^n S^2$. Since $\Z_2\cong \pi_2(B\SO(3))$, we can choose non-null homotopic continuous maps $g_i\colon S^2\to B\SO(3)$. We define a map $G\colon \bigvee_{i=1}^n S^2 \to B\SO(3)$ by $G= \bigvee_{i=1}^n g_i$.
%This induces a principal $\SO(3)$-bundle over $\bigvee_{i=1}^n S^2$.
Composing the attaching map $F$ with $G$ we get $[G\circ F]\in \pi_3(B\SO(3))=0$. Thus we can extend $G$ to $D^4$
and we get a continuous map $M\to B\SO(3)$, which induces a principal $\SO(3)$-bundle $E\to M$.
%, and obtain a principal $\SO(3)$-bundle $E\to M$.

Comparing the long exact sequence of homotopy groups of  $\SO(3)\to E\to M$ with the one of the classifying bundle $\SO(3)\to E\SO(3)\to B\SO(3)$, via the map $G\circ F$, we get  the following commutative diagram:
\begin{equation*}
\begin{tikzcd}
 \pi_2(E) \ar{r} & \pi_2(M) \ar{r}{\delta_2}\ar{d}{G_\ast} & \pi_1(\SO(3)) \ar{r}{i_\ast}\ar{d}{\cong} & \pi_1(E) \ar{r}\ar{d} & 1\\
0 \ar{r} & \pi_2(B\SO(3)) \ar{r}{\cong} & \pi_1(\SO(3)) \ar{r}{i_\ast} & 1 \ar{r} & 1\\
\end{tikzcd}
\end{equation*}
Since the map $G_\ast\colon \pi_2(M)\to \pi_2(B\SO(3))$ is surjective by construction, we see that the connecting homomorphism $\delta_2\colon \pi_2(M)\to\pi_1(\SO(3))$ is surjective. This implies that the induced map $i_\ast\colon \pi_1(\SO(3))\to \pi_1(E)$ is the trivial group homomorphism. Again, by Theorem~\ref{T:Obstructions to Covering fibrations} this principal $\SO(3)$-bundle cannot admit a covering fibration extending the cover $S^3\to \mathbb{R}P^3$.%the claim holds.
\end{ex}

%%%%%%%%%%%%%%%%%%%%%%%%%%%%%%%%
%
 \section{$\R P^n$-bundles}
%
%%%%%%%%%%%%%%%%%%%%%%%%%%%%%%%%

In light of Example~\ref{Ex: Example}, we show that the Main Result in \cite{FarrellGangKnopfOntaneda2017} still holds for smooth $\R P^n$-bundles with fiberwise pointwise strongly $1/4$-pinched metric, i.e.\ we prove Theorem~\ref{Theorem: Main theorem}. We begin by proving an analogous statement to Lemma~\ref{L: every round metric is isometric to round in a continuous way}

\begin{thm}\label{T: All round projective spaces are isometric to the standard one}
Let $g$ be a round metric on $\R P^n$ and let $\sigma$ denote the metric induced by the Riemannian covering $(S^n,\widetilde \sigma)\to (\R P^n, \sigma)$, where $\widetilde \sigma$ denotes the standard round metric on $S^n$. Then there is an isometry $\phi_g\colon (\R P^n, g) \to (\R P^n, \sigma)$. Moreover, $\phi_g$ depends on $g$ in a continuous way.
\end{thm}

\begin{proof}
Denote by $\widetilde{g}$ the lift of $g$ to $S^n$. Consider $N$ to be the \enquote{north pole} of $S^n$, and let $\{e_i\}$ be the standard orthonormal basis of $(T_N(S^n),\widetilde{\sigma})$ given by the inclusion $S^n\subset \R^{n+1}$.
Via the Gram-Schmidt algorithm, we obtain an orthonormal basis $\{\widetilde{e}_i\}$ on $(T_N(S^n),\widetilde{g})$. Denote by $i\colon T_N(S^n)\to T_N(S^n)$ the change of basis. Then, by Cartan's Theorem (see \cite[Chapter 8]{doCarmo}), there exists a unique isometry $\widetilde{\phi}\colon (S^n,\widetilde{g})\to (S^n,\widetilde{\sigma})$ such that $\widetilde{\phi}(N) = N$ and $\mathrm{D}_N \widetilde \phi (\widetilde{e}_i) = e_i$ for $i=1,2,\dots, n$.%\todo{check if this is already a global isometry, not just a local one}
%The isometry is defined as
%\[
%	\widetilde{\phi}(x) = \begin{cases} \exp_N^{\widetilde{\sigma}}_N(i(exp_N^{\widetilde{g}})^{-1}(x) & x\neq -N,\\
%	-p,& x=-N.
%
%						\end{cases}
%\]
%We observe that for $x\neq -N$, if $v= (\exp_N^{\widetilde{g}})^{-1}(x)$, then $-v = (\exp_{-N}^{\widetilde{g}})^{-1}(-x)$. Furthermore, since $-Id$ acts by isometries with respect to $\widetilde{\sigma}$, we have for $v\in T_N(S^n)$ that $-(\exp_N^{\sigma}(v) = \exp_{-N}^\sigma(-v)$. Putting everything together we obtain that

We now show that the isometry $\widetilde{\phi}$ is equivariant under the action of $\Z_2\cong \{Id,-Id\}$. To see this, we show that the cut locus of any point $p\in S^n$ consists of its antipodal point $-p$, with respect to the metric $\widetilde{g}$ as well as $\widetilde{\sigma}$.
Fix $p\in M$ and consider a shortest geodesic $c$ from $p$ to $-p$ with respect to the metric $\widetilde g$. Then the curve $(-Id)\circ c$ is, after reversing its direction, another shortest geodesic from $p$ to $-p$ with the same length. We observe that these two geodesics cannot be identical since otherwise the middle point of $c$ would be fixed by $-Id$.
The images of these curves via the isometry $\widetilde{\phi}$ give two different shortest geodesics from $\widetilde{\phi}(p)$ to $\widetilde{\phi}(-p)$ with respect to the standard round metric $\widetilde \sigma$. Clearly, this is only possible if $\widetilde{\phi}(-p) = -\widetilde{\phi}(p)$.

This implies $\widetilde \phi \circ (-Id) = (-Id)\circ \widetilde \phi$, so the isometry $\widetilde{\phi}$ is equivariant under the group action of $\Z_2$ and therefore induces an isometry $\phi\colon \R P^n \to \R P^n$.

We end the proof by showing that $\phi_g$ depends continuously on the metric $g$. Consider two metrics $g_1$ and $g_2$ on $\R P^n$ close to each other with respect to $C^\infty$ Whitney-topology. Since $\phi_{g_1}$ and $\phi_{g_2}$ are given by the exponential maps, as well as the Gram-Schmidt orthonormalization process with respect to $g_1$ and $g_2$, then the isometry $\phi_{g_1}$ is close to the isometry $\phi_{g_2}$ in $\Diff(\R P^n)$ with respect to the  $C^\infty$ Whitney topology (see \cite[Section~2.6]{CorroKordass2019}).

%in $\Gamma(Sym^2(\R P^n))$\todo{define this in the preliminaries, and the topologies},
%then the isometries (and thus diffeomorphisms) $\phi_{g_1}$ and $\phi_{g_1}$\todo{g2} given by Theorem~\ref{T: All round projective spaces are isometric to the standard one} are close in $\Diff(\R P^n)$ with respect to the Whitney topology. This is because $\phi_{g_i}$ is defined\todo{this is already in line 299} via the exponential maps and Grahm-Schmidt  \cite[Section~2.6]{CorroKordass2019}.%\todo{revise}.
%This isometry $\widetilde{f}$ also induces an isometry $\widetilde{f}\colon (D_+,\widetilde{g})\to (D_+,\widetilde{\sigma})$. Since the distance of $N$ and $-N$ with respect to $\sigma$ is $\pi$,  and the open upper hemisphere is the open ball of radius $\pi/2$ with respect to $\widetilde{\sigma}$ we have that the same holds with respect to $\widetilde{g}$. This induces a local isometry between $(\R P^n,g)\to (\R P^n,\sigma) $. To get a global isometry, we need only choose in $\partial D_+$ a unique representative of  each orbit. Assume that $p_1=\exp_p^{\widetilde{\sigma}}(\pi/2e_1)$. Set $q_1= f^{-1}(p_1)$. Since $\partial D_+=S^{n-1}$, we conclude that $\partial D_+/\Z_2=\R P^ {n-1}\subset \R P^n$. For $q_1$ and $p_1$ we can  domains in $D_+$, which are isometric via $\widetilde{f}$. We continue in this fashion, and obtain $D_1,D_2\subset D_+$ containing only one representative of each orbit, with $\widetilde{f}\colon (D_1,\widetilde{g})\to (D_2,\widetilde{\sigma})$ an isometry. This imples that we have an isometry $f\colon (\R P^n,g)\to (\R P^n,\sigma)$.
\end{proof}

With this we are able to give the proof of Theorem~\ref{Theorem: Main theorem}. Note that the proof is analogous to the proof of the Main Theorem in \cite{FarrellGangKnopfOntaneda2017}.%\todo{put idea of the proof and def of isos of bundles?}.

\begin{proof}[Proof of Theorem~\ref{Theorem: Main theorem}]
Let $p\colon E\to B$ be a smooth fiber bundle whose fibers are diffeomorphic to $\R P^n$. Assume that $E$ is equipped with a  pointwise strongly $1/4$-pinched fiberwise Riemannian metric $\{g_b\}_{b\in B}$. We will show that given any trivializing bundle charts, we can give transition maps which are isometries of the standard round metric. of $\R P^n$.

Take an atlas of trivializing bundle charts $\{(U_i,\alpha_i\}_{i\in J}$ of $E$, such that the closure $\bar{U}_i$ is compact (since $B$ is locally compact, such an atlas exists). Taking $U_{ij} = U_i\cap U_j$ and $\alpha_{ij}=\alpha_j\circ \alpha_i^{-1}\colon U_{ij}\times \R P^n \to U_{ij}\times \R P^n$, we have
\[
	\alpha_{ij}(b,v) = (b, \alpha_{ij}(b) (v) ),
\]
and by definition of a fiberwise Riemannian metric, for any $b\in U_{ij}$ the map $\alpha_{ij}(b) \colon (\R P^n,(\alpha_i)_\ast\, g_b )\to  (\R P^n,(\alpha_j)_\ast\, g_b )$ is an isometry.

By Proposition~\ref{P: pointwise stron 1/4-pinched flows to round metric} evolving the fiberwise metric via the normalized Ricci flow we obtain a  fiberwise Riemannian metric $\{g^{\bullet}_b\}_{b\in B}$ on $E$ such that each Riemannian metric $g^{\bullet}_b$ is a round metric. Thus we obtain continuous maps $G_i\colon U_i \to \Met^1(\R P^n)$ given by
\[
	G_i(b)  = (\alpha_{i})_\ast\, g_b^\bullet.
\]
We now consider the continuous map $\phi\colon \Met^{1}(\R P^n)\to \Diff(\R P^n)$ given by Theorem~\ref{T: All round projective spaces are isometric to the standard one}. Composing with $G_i$ we obtain a continuous map $\phi\circ G_i\colon U_i\to \Diff(\R P^n)$. We take the diffeomorphism $f_{i}\colon U_{ij}\times \R P^n\to  U_{ij}\times \R P^n$ defined as
\[
	f_i(b,v) = (b,\phi\circ G_i (b)(v)).
\]
For any pair of indices $i,j\in J$ we define the diffeomorphism $\beta_{ij}\colon  U_{ij}\times \R P^n\to U_{ij}\times \R P^n$ as
\[
\beta_{ij}(b,v) = f_j\circ\alpha_{ij}\circ f^{-1}_i (b,v).
\]
Thus by construction the following diagram commutes:
\begin{center}
\begin{tikzcd}
U_{ij}\times \R P^n \ar{r}{f_{i}} \ar{d}{\alpha_{ij}}    & U_{ij}\times \R P^n  \ar{d}{\beta_{ij}}\\
U_{ij}\times \R P^n \ar{r}{f_{j}} & U_{ij}\times \R P^n
\end{tikzcd}
\end{center}
This implies that $p\colon E\to B$ is isomorphic as a fiber bundle to  the smooth fiber bundle $p'\colon E'\to B$ associated to the transition maps $\beta_{ij}$.

Fix $b\in U_{ij}\subset B$. Then
\[
	\beta_{ij} (b) = \phi_{(\alpha_j)_\ast\, g_b^\bullet} \circ \alpha_{ij}(b)\circ\left(\phi_{(\alpha_i)_\ast\, g_b^\bullet} \right)^{-1}\colon (\R P^n,\sigma)\to (\R P^n,\sigma)
\]
is by construction an isometry of the standard round metric $\sigma$. Thus the structure group of $p'\colon E'\to B$ is contained in the isometry group of $\sigma$, as claimed.
\end{proof}

%%%%%%%%%%%%%%%%%%%%%%%%%%%%%%%%
%

%%%%%%%%%%%%%%%%%%%%%%%%%%%%%%%%
%
% THE BIBLOGRAPHY
%
%%%%%%%%%%%%%%%%%%%%%%%%%%%%%%%%

\bibliographystyle{siam}

\nocite{Brendle2010}
\bibliography{Bibliography.bib}

\begin{thebibliography}{10}

\bibitem{AndrewsHopper}
{\sc B.~Andrews and C.~Hopper}, {\em The {R}icci flow in {R}iemannian
  geometry}, vol.~2011 of Lecture Notes in Mathematics, Springer, Heidelberg,
  2011.
\newblock A complete proof of the differentiable 1/4-pinching sphere theorem.

\bibitem{BamlerKleiner2017}
{\sc R.~Bamler and B.~Kleiner}, {\em Ricci flow and diffeomorphism groups of
  3-manifolds},
  \href{https://arxiv.org/pdf/1712.06197.pdf}{arXiv:1712.06197~[math.DG]},
  (2017).

\bibitem{BamlerKleiner2019}
\leavevmode\vrule height 2pt depth -1.6pt width 23pt, {\em Ricci flow and
  contractibility of spaces of metrics},
  \href{https://arxiv.org/pdf/1909.08710.pdf}{arXiv:1909.08710~[math.DG]},
  (2019).

\bibitem{BeckerGottlieb1973}
{\sc J.~C. Becker and D.~H. Gottlieb}, {\em Coverings of fibrations},
  Compositio Math., 26 (1973), pp.~119--128.

\bibitem{Brendle2010}
{\sc S.~Brendle}, {\em Ricci flow and the sphere theorem}, vol.~111 of Graduate
  Studies in Mathematics, American Mathematical Society, Providence, RI, 2010.

\bibitem{BrendleSchoen2009}
{\sc S.~{Brendle} and R.~{Schoen}}, {\em {Manifolds with \(1/4\)-pinched
  curvature are space forms.}}, {J. Am. Math. Soc.}, 22 (2009), pp.~287--307.

\bibitem{CorroKordass2019}
{\sc D.~Corro and J.-B. Korda\ss}, {\em Short survey on the existence of slices
  for the space of riemannian metrics},
  \href{https://arxiv.org/pdf/1904.07031.pdf}{arXiv:1904.07031~[math.DG]},
  (2019).

\bibitem{deRahm1950}
{\sc G.~de~Rham}, {\em Complexes \`a automorphismes et hom\'{e}omorphie
  diff\'{e}rentiable}, Ann. Inst. Fourier Grenoble, 2 (1950), pp.~51--67
  (1951).

\bibitem{doCarmo}
{\sc M.~P.~a. do~Carmo}, {\em Riemannian geometry}, Mathematics: Theory \&
  Applications, Birkh\"{a}user Boston, Inc., Boston, MA, 1992.

\bibitem{FarrellGangKnopfOntaneda2017}
{\sc T.~Farrell, Z.~Gang, D.~Knopf, and P.~Ontaneda}, {\em Sphere bundles with
  {$1/4$}-pinched fiberwise metrics}, Trans. Amer. Math. Soc., 369 (2017),
  pp.~6613--6630.

\bibitem{GompfStipsicz1999}
{\sc R.~E. Gompf and A.~I. Stipsicz}, {\em {$4$}-manifolds and {K}irby
  calculus}, vol.~20 of Graduate Studies in Mathematics, American Mathematical
  Society, Providence, RI, 1999.

\bibitem{Hamilton1982}
{\sc R.~S. Hamilton}, {\em Three-manifolds with positive {R}icci curvature}, J.
  Differential Geom., 17 (1982), pp.~255--306.

\bibitem{Hatcher1983}
{\sc A.~E. Hatcher}, {\em A proof of the {S}male conjecture, {${\rm
  Diff}(S\sp{3})\simeq {\rm O}(4)$}}, Ann. of Math. (2), 117 (1983),
  pp.~553--607.

\bibitem{Hirsch}
{\sc M.~W. Hirsch}, {\em Differential topology}, vol.~33 of Graduate Texts in
  Mathematics, Springer-Verlag, New York, 1994.

\bibitem{KrieglMichor}
{\sc A.~Kriegl and P.~W. Michor}, {\em The convenient setting of global
  analysis}, vol.~53 of Mathematical Surveys and Monographs, American
  Mathematical Society, Providence, RI, 1997.

\bibitem{RosenbergStolz}
{\sc J.~Rosenberg and S.~Stolz}, {\em Metrics of positive scalar curvature and
  connections with surgery}, in Surveys on surgery theory, {V}ol. 2, vol.~149
  of Ann. of Math. Stud., Princeton Univ. Press, Princeton, NJ, 2001,
  pp.~353--386.

\bibitem{Smale1959}
{\sc S.~Smale}, {\em Diffeomorphisms of the {$2$}-sphere}, Proc. Amer. Math.
  Soc., 10 (1959), pp.~621--626.

\bibitem{TuschmannWraith}
{\sc W.~Tuschmann and D.~J. Wraith}, {\em Moduli spaces of {R}iemannian
  metrics}, vol.~46 of Oberwolfach Seminars, Birkh\"{a}user Verlag, Basel,
  2015.

\end{thebibliography}
%\printbibliography

%\todo[inline]{look over references (formatting, missing fields)}
\end{document}